\documentclass{amsart}

\usepackage{graphicx, amsmath, amssymb,amsthm, geometry, caption}
\usepackage{indentfirst}
\usepackage{graphicx}
\usepackage{amsrefs}

\issueinfo{}
  {}
  {}
  {}
\pagespan{1}{}
\copyrightinfo{}
  {Korean Mathematical Society}

\allowdisplaybreaks

\theoremstyle{plain}

\newtheorem{theorem}{Theorem}[section]
\newtheorem{proposition}[theorem]{Proposition}
\newtheorem{lemma}[theorem]{Lemma}
\newtheorem{corollary}[theorem]{Corollary}

\theoremstyle{definition}

\newtheorem{example}[theorem]{Example}

\theoremstyle{remark}
\newtheorem{remark}[theorem]{Remark}

\def\*#1{\mathbf{#1}} 
\newtheorem*{thma}{\textbf{Theorem A}}
\newtheorem*{thmb}{\textbf{Theorem B}}

\begin{document}
\title[Theorems of Pappus-Guldin]
{Generalizations of the theorems of Pappus-Guldin in the Heisenberg groups}

\author[Y.C. Huang]{Yen-Chang Huang}
\address{Yen-Chang Huang \\ Department of Applied Mathematics\\National University of Tainan \\Tainan, Taiwan}
\email{ychuang@mail.nutn.edu.tw}


\subjclass[2010]{Primary: 53C17, Secondary: 53C65, 53C23}
\keywords{Pappus-Guldin Theorem, Sub-Riemannian manifolds, Pseudo-hermitian geometry}

\begin{abstract}
In this paper we study areas (called \textit{p-areas}) and volumes for parametric surfaces in the 3D-Heisenberg group $\mathbb{H}_1$, which is considered as a flat model of pseudo-hermitian manifolds. We derive the formulas of p-areas and volumes for parametric surfaces in $\mathbb{H}_1$ and show that the classical result of Pappus-Guldin theorems for surface areas and volumes hold if the surfaces satisfy some geometric properties. Some examples are also provided, including the surfaces with constant p-mean curvatures.
\end{abstract}

\maketitle

\section{Introduction}
A remarkable formula for the volumes and surface areas of solid objects in $\mathbb{R}^3$ was given by A. Pappus about AD 300 and was rediscovered by P. Guldin in 1641 in many Calculus books. Nowadays the theorem is known as Pappus-Guldin theorem or Pappus theorem. We refer the interested readers to \cite{goodman, GAM} about a short historical review of the development of the theorem, including the early generalizations by Euler and Richter. The papers \cite{bulmer} and \cite{radelet} mentioned the reasons why mathematicians in the 17th century did not know the work of Pappus.

The Pappus-Guldin theorem states the method of finding volumes and surface areas respectively for any solid of revolution into two parts:

\begin{thma}\label{ThmA}
The volume of a solid of revolution $M$ generated by rotating a region $\mathcal{R}$ about a line $\ell$ that does not meet $\mathcal{R}$ is given by
$$Volume(M)= Area(\mathcal{R})\times s_0,$$
where $s_0$ is the perimeter of the circle described by the centroid of $\mathcal{R}$ during the rotation.
\end{thma}

\begin{thmb}\label{ThmB}The surface area of a surface of revolution $\Sigma$ generated by rotating a plane curve $\mathcal{C}_0$ about an axis $\ell$ that does not meet $\mathcal{C}_0$ is given by
$$Area(\Sigma)=Length(\mathcal{C}_0)\times s_0,$$
where $s_0$ is the length of the curve described by the centroid of $\mathcal{C}^0$ during the rotation.
\end{thmb}

In \cite{goodman} A.W. Goodman and G. Goodman gave a proof of the generalization of Theorem A for domain $M$ generated by moving a plane region $\mathcal{R}$ around an \textit{arbitrary} space curve $\gamma$. When the center of mass of $\mathcal{R}$ is on $\gamma$, they have
$$Volume(M) = Area(\mathcal{R})\times Length(\gamma).$$
However, the authors showed that Theorem B holds only for the surface area of $\Sigma$ generated by moving a plane curve $\mathcal{C}_0$ around a \textit{plane curve} $\gamma$ with "natural motions" (see \cite[p358]{goodman}). After that, Flanders \cite{Flanders} and Pursell \cite{Pursell} supplemented the results of surface areas in \cite{goodman} by showing that there exists a unique spin function such that the area of the surface $\Sigma$ generated by $\mathcal{C}_0$, when $\mathcal{C}_0$ moves with the spin, agrees with the result of Theorem B
$$Area(\Sigma)=Length(\mathcal{C}_0)\times Length(\gamma).$$

On one hand, Pursell used the classical method of differential geometry and showed that the "natural motion" in \cite{goodman} is actually a motion with $\mathcal{C}_0$ fixed to a Frenet frame. On the other hand, Flanders obtained the result of Theorem B in a more efficient method by using moving frames and differential forms. In addition, Gray-Miquel \cite{GM} also generalized Pappus-type theorems for volumes and surface areas to the hypersurfaces in the space forms. We also stress that there is a natural connection between Pappus-Guldin formula and the works by H. Weyl. In \cite{weyl1939volume} Weyl proved the formulas of the volumes of a tube $P_r$ with radius $r$ and the corresponding tubular hypersurface $\partial P_r$ around a $q$-dimensional submanifold $P$ in a Euclidean or Spherical space, and finally those formulas were used later in the proof of the generalized Gauss-Bonnet Theorem by Allendoerfer \cite{allend} and Fenchel \cite{fenchel}. We refer the reader to the book \cite{weyl1939volume} and the survey \cite{vanhecke} by Vanhecke for more details about the connection of Weyl's works and Pappus-Guldin theorem.

In the last years the study of variational problems in sub-Riemannian geometry and pseudo-hermitian geometry have received an increasing interest. Especially the desire to achieve a better understanding of questions of global geometry involving the areas and volumes, such as the Plateau problem or the isoperimetric problem, has motivated the recent development of a theory of Pappus-Guldin theorem. More precisely, in the present paper we ask the following questions:\\

\textit{Does Pappus-Guldin theorem hold in the Heisenberg group $\mathbb{H}_1$? What are the appropriate geometric quantities (lengths, surface areas, volumes) satisfying the geometry of $\mathbb{H}_1$ as a standard model of sub-Riemannian manifolds?}\\

The answer is affirmative if we take the following geometric quantities: (1) horizontal lengths for curves, (2) p-area for surface areas (defined in next section), (3) Euclidean volumes for solid domain in $\mathbb{H}_1$. We shall show the generalization of Pappus-Guldin theorem for volumes and surface areas of any tube consisting of a plane region moving along a curve passing through the region.

In \cite[p 385]{goodman} the authors considered the tubes in $\mathbb{R}^3$ constructed by the following method: given a $C^3$-closed curve
$$\textbf{R}(s)=\big(x(s), y(s), z(s)\big)$$
for $s\in [0, s_0]$ and a Frenet frame $\{\textbf{T}(s), \textbf{N}(s), \textbf{B}(s)\}$ defined on $\textbf{R}(s)$, where the unit vectors $\textbf{T}, \textbf{N}, \textbf{B}$ are the tangent, the normal, and the binormal vectors, respectively. By setting $x(0)>0$, $y(0)=z(0)=0$, and the initial point $\textbf{R}(0)=\big(x(0),y(0),z(0)\big)=\big(x(s_0),y(s_0),z(s_0)\big)\in \textbf{R}(s)$, the authors also defined the plane curve by
$$\mathcal{C}_0: \textbf{R}_1(t)=\big(f(t), g(t), 0\big),$$
for some $C^1$-functions $f(t),g(t)$ defined on $[0, t_0]$. Note that the point $\textbf{R}(0)$ is regarded as the origin for describing the curve $\mathcal{C}_0$. Thus, the plane region $\mathcal{R}\ni \textbf{R}(0)$ enclosed by $\mathcal{C}_0$ lies on the $xy$-plane (as shown in Fig. \ref{fig1}). The "natural motion" of the region $\mathcal{R}$ along the curve $\textbf{R}(s)$ is defined by choosing the coordinate system so that at $s=0$, the vectors $\textbf{T}$ and $\textbf{N}$ lie on the $xz$-plane, and $\textbf{N}$ lies on the $x$-axis pointed toward the origin. When the region $\mathcal{R}$ moves along $\textbf{R}(s)$ starting from $\textbf{R}(0)$, $\mathcal{R}$ is always on the plane spanned by $\textbf{N}(s)$ and $\textbf{B}(s)$, and the line segment contained in $\textbf{B}(0)$ at $\textbf{R}(0)$ is always contained in $\textbf{B}(s)$ for any $s$. More precisely, the surface $\Sigma$ generated by moving the plane region $\mathcal{R}$ along the curve $\textbf{R}(s)$ can be represented by the parametrized surface
\begin{align}\label{goodman2}
\mathbb{X}(s,t)=\textbf{R}(s)-f(t)\textbf{N}(s)-g(t)\textbf{B}(s),
\end{align}
Note that when the region $\mathcal{R}$ moves along the curve $\textbf{R}(s)$, $\mathcal{R}$ is always perpendicular to the tangent vector $\textbf{T}$.

\begin{figure*} [h!]
    \centering
    \includegraphics[width = .8\linewidth]{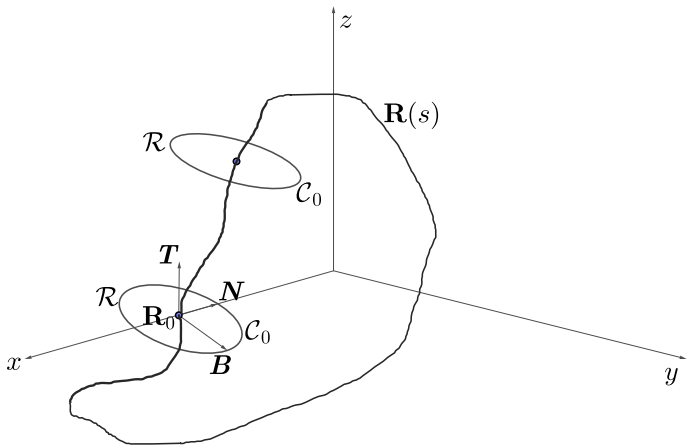}
    \caption{"Natural motion" of the region $\mathcal{R}$ along the closed curve \textbf{R}(s)}
    \label{fig1}
\end{figure*}

Next we begin to set up for this case. Regarding 3-dimensional Heisenberg group $\mathbb{H}_1$ as the standard model of the pseudo-hermitian manifolds, there exists the contact planes $\xi_p$ and the CR structure $J$ defined for all point $p\in \mathbb{H}_1$. For some fundamental background about $\mathbb{H}_1$, we refer the readers to the next section. Unlike the surfaces in \cite{goodman} that the authors only consider the surfaces generated by the vectors $\textbf{N}$ and $\textbf{B}$ in \eqref{goodman2}, in the 3-dimensional Heisenberg group $\mathbb{H}_1$ we consider a more generalized surface
\begin{align}\label{maincurve}
\Sigma: \mathbb{X}(s,t)=\gamma(s)+f(t)U(s)+g(t)V(s)+h(t)T,
\end{align}
for some real-valued functions $f(t), g(t), h(t)$ with suitable regularity. Here $\gamma(s)$ is a horizontally regular curve with horizontal arc-length, i.e., the orthogonal decomposition of the velocity vector $\gamma'(s) =\gamma'_\xi (s) +\gamma'_T(s)$ has the nonzero contact part $\gamma'_\xi(s)\in \xi_{\gamma(s)}$ and $|\gamma'_\xi|=1$ w.r.t. the Levi-metric $\langle X,Y\rangle =\frac{1}{2}d\Theta(X, JY)$ for any vectors $X,Y$ in $\mathbb{H}_1$, and $U(s)=\gamma_\xi(s), V(s)=J \gamma_\xi(s)$, $T=\frac{\partial}{\partial z}$, and $J$ is the CR structure.
Note that $\{U(s), V(s), T\}$ forms an orthonormal basis w.r.t. the Levi-metric. A more natural perspective for the setting of \eqref{maincurve} is that when considering $\mathbb{H}_1$ as the Lie group with the group operation defined by \eqref{groupop} (see next section), the surface $\Sigma$ can be obtained by moving the frame $\{U(0), V(0), T\}$ at the origin along the curve $\gamma(s)$. More precisely, the parametrization of \eqref{maincurve} can be equivalently defined by
\begin{align}\tag{2'}\label{maincurve2}
\Sigma: \mathbb{X}(s,t)=L_{\gamma(s)}\big(f(t)U(0)+g(t)V(0)+h(t)T\big),
\end{align}
where $U(0)$ and $V(0)$ are two fixed unit vectors starting from the origin and ending at somewhere on the $xy$-plane. Indeed, suppose $\gamma(0)=0\in \mathbb{H}_1$ and since $|\gamma'_\xi(s)|=1$, we may set $U(s)=\eta \mathring{e}_1(\gamma(s))+\zeta \mathring{e}_2(\gamma(s))$ and $V(s)=-\zeta \mathring{e}_1(\gamma(s))+\eta \mathring{e}_2(\gamma(s))$ for some real numbers with $\eta^2+\zeta^2=1$. Here $\mathring{e}_1, \mathring{e}_2$ are the standard basis of the invariant vector fields defined by \eqref{standardbasis}. In particular, $U(0)=(\eta, \zeta, 0)$ and $V(0)=(-\zeta, \eta, 0)$. By using the group operation \eqref{groupop} one has
\begin{align*}
L_{\gamma(s)}\big(f(t)U(0)&+g(t)V(0)+h(t)T\big)\\
&=\big(x(s),y(s), z(s)\big)\cdot \big(f(t) \eta-g(t)\zeta, f(t)\zeta +g(t) \eta, h(t)\big) \\
&=\gamma(s)+f(t)\Big[\eta \mathring{e}_1(\gamma(s))+\zeta\mathring{e}_2(\gamma(s))\Big]+g(t)\Big[-\zeta \mathring{e}_1(\gamma(s))+\eta \mathring{e}_2(\gamma(s))\Big]+h(t)T \\
&=\gamma(s)+f(t) U(s)+g(t) V(s) +h(t) T,
\end{align*}
and so \eqref{maincurve} and \eqref{maincurve2} are equivalent.

Analogous to the settings in \cite{goodman}, the surface of revolution $\Sigma$ is defined by rotating the closed space curve
$$\mathcal{C}_s(t): f(t)U(s)+g(t)V(s)+h(t)T$$
for some fixed $s\in [0, s_0]$ along the closed curve
$$\gamma(s)=\big(x(s),y(s), z(s)\big).$$

Throughout this paper we always use the parameter $s\in [0, s_0]$ for the curve $\gamma$ and the frames $U, V$ along $\gamma$, and use the parameter $t\in [0,t_0]$ for the functions $f,g,h,$ describing the curve $\mathcal{C}_s$. For simplicity, we shall also drop out the parameters in derivatives if the notion is clear, for instance, $f'U+gV'$ means $\frac{df(t)}{dt}U(s)+g(t)\frac{dV(s)}{ds}$. We often denote $M$ by the solid of revolution and $\partial M$ or $\Sigma$ by its surface. Also, in the present paper we only consider the surface $\mathbb{X}(s,t)$ defined by \eqref{maincurve} satisfying these fundamental assumptions:\\

\begin{enumerate}
\item[(\textbf{A1})] The closed curve $\mathcal{C}_s(t)$ does not intersect itself when moving along the curve $\gamma(s)$.
\item[(\textbf{A2})] The domain enclosed by the parametrized surface $\Sigma$ defined by \eqref{maincurve} is not self-intersected.
\end{enumerate}

For the regularity of the functions $\gamma$, and $f,g,h$, we have the minimal assumptions for the order of derivatives necessary to carry out the proofs. It turns out that the requirements are different for the two curves $\gamma$ and $\mathcal{C}_s$. For $\gamma(s)$ we require that the third derivatives $\gamma^{'''}$ be continuous in the interval $0\leq s\leq s_0$. We emphasis that $\gamma$ may not necessarily be closed, but if it \textit{is} closed we have to assume that $\gamma^{(k)}(0)=\gamma^{(k)}(s_0)$ for $k=0,1,2,3,$ where $\gamma^{(k)}$ is the $k$-th derivative with respect to the horizontal arc-length $s$. For the closed curve $\mathcal{C}_s$ (when fixed $s$), in the first glance, it seems that we need to have the functions $f(t), g(t), h(t)$ be $C^1$-functions in the interval $0\leq t\leq t_0$, except for Corollary \ref{minimalsurface} that we need a $C^2$-assumption for those functions. However, a stronger regularity assumption is needed to avoid the issue of singular points. A point on a surface in $\mathbb{H}_1$ is called \textit{singular} if the contact plane and tangent plane at that point coincide with each other; otherwise, it is called \textit{regular} (see next section for more detail). According to \cite[Theorem 3.3 or Theorem B]{CHMY}, a singular point on a $C^2$-smooth surface is either isolated or passed through by a $C^1$-smooth singular curve in a neighborhood if the p-mean curvature is bounded. Since the p-mean curvature needs $C^2$-regularity for the surface, throughout the whole paper we always assume that the functions $f(t), g(t), h(t)$ are of $C^2$-regularity. Thus, the set of singular points on the constructed surfaces has measure zero with respect to the area and volume measures and so we can ignore the influence of the singular point when considering the surface areas and volumes.

Next we introduce the appropriate geometric quantities for Pappus-Guldin Theorem. Let $M$ be a domain with boundary $\partial M$ in $\mathbb{H}_1$. For any point $p\in \partial M$ is not a singular point, denote by the unit vector $e_1(p)=TM_p\cap \xi_p$, $e_2=Je_1$, and the standard contact form $\Theta=dz+xdy-ydx$, where $J$ is the CR structure, $J^2=-1$. The set $\{e_1, e_2, T\}$ forms an standard orthonormal frame with the corresponding coframe $\{e^1, e^2, \Theta\}$. As we know that the appropriate notions for volumes and areas in $\mathbb{H}_1$ are the Euclidean volumes and the p-areas, respectively, and the p-area can be obtained by the variation of the volume form along the $e_2$-direction (equation (2.5) in \cite[p135]{CHMY}). More precisely, those geometric quantities are defined by
\begin{align}\label{parea0}
Volume(M)&:=V(M)=\frac{1}{2}\int_M \Theta\wedge d\Theta, \\
p\text{-}area(\partial M)&:=\mathcal{A}(M)=\int_{\partial M} \Theta \wedge e^1, \nonumber
\end{align}
where $\Theta\wedge e^1$ is called the \textit{p-area form}. Note that the volumes and the p-areas are invariant under the "rigid motions" in $\mathbb{H}_1$, namely, they are invariant under the group of pseudo-Hermitian transformations in $\mathbb{H}_1$, $PHS(1)$, which consists of the left translations and the elements in the unitary group $U(1)$. For the details of this group $PSH(n)$, we refer the reader to \cite{CHL} for $n=1$ and \cite{chiu2015fundamental} for $n\geq 1$; both papers are the first published papers studying on the fundamental theorems for curves and hypersurfaces in the Heisenberg groups.

A key observation for the study of Pappus-Guldin theorem is shown in Proposition \ref{pareaform1}: for any parametric surface $\Sigma:\mathbb{X}(s,t)=\big(X(s,t), Y(s,t), Z(s,t)\big)$ in $\mathbb{H}_1$ we derive the p-area form $d\Sigma_p$ at the nonsingular point $p\in\Sigma$, namely,
\begin{align}\label{pareaform5}
d\Sigma_p = \sqrt{A^2+B^2}ds dt,
\end{align}
where
\begin{align*}
A&:=Z_s X_t- X_s Z_t +X(X_tY_s-X_sY_t),  \\
B&:=Z_s Y_t- Y_s Z_t+ Y(X_tY_s -X_sY_t),
\end{align*}and the subscripts mean the partial derivatives with respect to the specific variables.
By using the divergence Theorem (Proposition \ref{Divergence}), the volume enclosed by the surface $\Sigma$ can be represented by the integral over $\Sigma$. Moreover, when $\Sigma$ is generated by the curve $\gamma(s)$ and the functions $f(t), g(t), h(t)$ as defined by \eqref{maincurve}, we obtain the first result showing that the volume of an enclosed tube-shaped domain can be obtained by the information from its surface.
\begin{theorem}\label{mainvolume}
Let $M$ be a domain in the 3-dimensional Heisenberg group $\mathbb{H}_1$ with boundary $\partial M$, described by the parametrized surface $\mathbb{X}(s,t)=\gamma(s) +f(t) U(s)+g(t) V(s) +h(s) T$ for any $(s,t)\in [0,s_0]\times [0,t_0]$, where $\gamma(s)$ is a horizontally regular curve (may not be necessarily closed), $U(s)=\gamma'_\xi$, and $V(s)=J\gamma'_\xi$. If $M$ satisfies the conditions $(A1), (A2)$, then the volume of $M$ is given by
\begin{align}
Volume(M)&=\iint \kappa (xx'+yy')h'f +\iint (yx'-y'x)\kappa h'g
+\iint \kappa h'f^2  -\iint \tau g'f  \\
&+\iint \kappa g^2 h' +\iint(2\tau-z') f'g+s_0\cdot  \int_0^{t_0}\Big( gg'f - h'g - f'g^2 \Big)dt ,\nonumber
\end{align}
where all double integrals are taken over the domain $[0,s_0]\times [0,t_0]$.
\end{theorem}

\begin{remark}
Recently Chiu-Ho (\cite{chiu2019global}, Theorem 2.1) defined the Gauss map $G:\gamma\rightarrow \mathbb{S}^1\subset \xi_0(0)$ by $G(p)=L_{p^{-1}*}e_2$ for horizontally regular closed cure $\gamma$ in $\mathbb{H}_1$. They also showed that the degree of the Gauss map satisfying $deg G=\frac{1}{2\pi}\int_0^{s_0} \kappa(s)ds$. Therefore, the third and eighth terms in Theorem \ref{mainvolume} can be written in terms of the topological information of $M$ by $\iint \kappa h' f^2= 2\pi deg  G \int_0^{t_0} h'f dt$ and $\iint \kappa g^2 h' = 2\pi deg  G \int_0^{t_0}g^2h'dt$, respectively.
\end{remark}

By the definition \eqref{pcur} of the p-curvatures, the assumption $\kappa=0$ for the curve $\gamma$ is equivalent to that the projection of $\gamma$ onto the $xy$-plane is a line. The following result shows that when $\kappa(s)\equiv 0$ and $f(t)\equiv 0$, the volume of the enclosed tube-shaped domain satisfies Pappus-Guldin theorem.

\begin{corollary}[Pappus-Guldin Theorem for volumes]\label{mainvolume2}
Let $\gamma(s)$ be a horizontally regular curve parametrized by horizontal arc-length with p-curvature $\kappa(s)\equiv 0$. Suppose the parametrized surface $\mathbb{X}(s,t)$, given by moving the plane region $\mathcal{R}$ enclosed by the plane curve $\mathcal{C}_0(t)=\gamma(0)+g(t)U(0)+h(t)T$ along $\gamma$ is defined as
$$\mathbb{X}(s,t)=\gamma(s)+g(t)U(s) + h(t)T$$
for any $(s,t)\in [0, s_0]\times [0,t_0]$ satisfying the assumptions $(A1)$, $(A2)$. Then the volume of enclosed domain $M$ by $\mathbb{X}(s,t)$ satisfies Pappus-Guldin theorem, namely,
\begin{align*}
V(M)=\frac{1}{2}\cdot \text{\textit{horizontal length}}(\gamma)\cdot \mathcal{A}(\mathcal{R}),
\end{align*}where $\mathcal{A}(R)$ is the p-area of the area of $\mathcal{R}$.
\end{corollary}

Our second result shows that the p-area form of the parametrized surface $\Sigma$ defined by \eqref{maincurve} can be achieved in terms of $\gamma, f, g, h$ and their derivatives by substituting the surface into \eqref{pareaform5}.

\begin{theorem}\label{mainarea}
Given a horizontally regular curve $\gamma(s)=(x(s), y(s), z(s))$ parametrized by the horizontal arc-length $s$ with the p-curvature $\kappa$ and contact normality $\tau$ defined by \eqref{pcur}, \eqref{tau} respectively in the Heisenberg group $\mathbb{H}_1$. For any $C^1$-functions $f(t), g(t), h(t)$, suppose the parametrized surface $\Sigma$ defined by \eqref{maincurve}
\begin{align*}
\mathbb{X}: (s,t)\in [0,s_0]\times [0, t_0]\mapsto \gamma(s) + f(t)U(s) + g(t)V(s) +h(t)T,
\end{align*} satisfying the assumption $(A1), (A2)$ and $p$ is a nonsingular point on $\Sigma$. Then the p-area form $d\Sigma_p$ of the surface $\Sigma$ at $p$ is given by
\begin{align}\label{pareaform2}
d\Sigma_p&= \\
\Big[(h')&^2\big((\kappa f)^2 +(1-\kappa g)^2 \big)-2h'(\tau-g)\big(\kappa(fg'-f'g)+f'\big)+(\tau-g)^2\big( (f')^2+(g')^2\big)\Big]^{1/2} dsdt. \nonumber
\end{align}
\end{theorem}

Besides, we also have Pappus-Guldin theorem for surface areas if we make suitable assumptions on the functions $\gamma, f,g$, and $h$.

\begin{corollary}[Pappus-Guldin Theorem for surface areas]\label{pappsarea}
Given the surface defined by
$$\Sigma: \mathbb{X}(s,t)=\gamma(s)+f(t)U(s)+g(t)V(s)+h(t)T$$ in $\mathbb{H}_1$ satisfying the assumptions $(A1), (A2)$ such that the curve $\gamma(s)$ is parametrized by horizontal arc-length $s\in [0, s_0]$ and $f,g,h$ are functions defined on $[0,t_0]$.
\begin{enumerate}
\item If $h(t)\equiv C$ for some constant $C$ and $(f')^2+ (g')^2=1$, the p-area of $\Sigma$ is given by
\begin{align}\label{pappasarea1}
\mathcal{A}(\Sigma)= \int_0^{s_0}\int_0^{t_0} |\tau(s)-g(t)| dt ds.
\end{align}
In addition, if $\tau\equiv 1>g $ on $[0,s_0]\times [0,t_0]$ (resp. $\tau\equiv -1 <g$), then Pappus-Guldin theorem holds, namely,
\begin{align}\label{pappasarea2}
\mathcal{A}(\Sigma)&= s_0t_0-s_0 \int_0^{t_0}g(t) dt.   \\
 (\text{resp. } &=s_0t_0+s_0 \int_0^{t_0} g(t) dt) \nonumber
\end{align}
\item If $\kappa \equiv 0, \tau\equiv 1, g\equiv 0$, and $h(0)=h(t_0)$, then Pappus-Guldin theorem also holds, namely,
\begin{align}\label{ktgh}
\mathcal{A}(\Sigma)=s_0 \int_0^{t_0}|f'(t)|dt.
\end{align}
\end{enumerate}
\end{corollary}

Note that in Corollary \ref{pappsarea} (1) the assumption $h(t)\equiv 0$ means that the curve $\mathcal{C}_s$ is on the contact plane $\xi_{\gamma(s)}$ for all $s\in [0,s_0]$; also in the second part of the corollary, when $g\equiv 0$, the horizontal length of the curve $\mathcal{C}_s$ is $\int_0^{t_0}|f'(t)|dt$ (see Lemma \ref{horizontallengthdetermined}), and so the formula \eqref{ktgh} is again the form satisfying Pappus-Guldin theorem. In Section \ref{mainareapappsarea} we show different types of examples (Example \ref{examplereqornotone}, Example \ref{corollaryexample2}) satisfying conditions in Corollary \ref{pappsarea}.

The paper is organized as follows. First, some basic results of horizontally regular curves and surfaces will be reviewed in Section \ref{Preliminary}, including the geometric invariants: the p-curvature $\kappa$ and the contact normality $\tau$, and the Frenet frame formula of curves in $\mathbb{H}_1$. In Section
\ref{pareavolumeexammples} we derive the p-area forms for the parametric surfaces. The volumes of domains enclosed by tube-shaped surfaces are also obtained by the divergence theorem, which is proved in terms of the pseudo-hermitian connection on Cauchy-Riemannian manifolds. An isoperimetric-type inequality is also provided. We also provide some examples (Example \ref{ex1}, Example \ref{ex2}) showing that when the p-areas coincide with the usual Euclidean-areas. In Section \ref{mainvolumemainvolum2} we prove Pappus-Guldin theorem for volumes (Theorem \ref{mainvolume}, Corollary \ref{mainvolume2}). Finally, in Section \ref{mainareapappsarea} first we show some special cases for cylindrical surfaces (Proposition \ref{observ1}, Proposition \ref{observ2}) as the key observation to the main theorem. Secondly, the proofs of Pappus-Guldin theorem for surfaces (Theorem \ref{mainarea} and Corollary \ref{pappsarea}) are provided. We also give some examples to support our theorems, including the surfaces with constant p-mean curvatures.

\textbf{Acknowledgement} The author would like to thank the anonymous reviewer for the comments for the regularity of the constructed surfaces. This work was funded in part by National Center for Theoretical Sciences (NCTS) in Taiwan and in part by Ministry of Science and Technology, Taiwan, with grant Number: 108-2115-M-024-007-MY2.

\section{Preliminary}\label{Preliminary}
We recall some terminologies for our purpose. For more details about the Heisenberg groups, we refer the readers to
\cites{CHMY, CHL, goodman, Pursell, GM}. The author also noted that there are some applications of image processing using the Heisenberg groups as the configuration spaces and dealing with the problems for image completion, for examples, see \cite{2008implementation, citti2016sub} and the references therein. The $3$-dimensional Heisenberg group $\mathbb{H}_1$ is the Lie group $(\mathbb{R}^3, \cdot)$ with the group operation $\cdot$ defined by
\begin{align}\label{groupop}
(x,y,z)\cdot (x',y',z')=(x+x',y+y',z+z'+yx'-xy'),
\end{align}
for any point $(x,y,z)$, $(x',y',z')\in\mathbb{R}^3$.
For any point $p\in\mathbb{H}_1$, the \textit{left translation} by $p$ is the diffeomorphism $L_p(q):=p\cdot q$. The standard basis of left invariant vector fields (i.e., invariant under any left translation) is given by
\begin{align}\label{standardbasis}
\mathring{e}_1(p):=\frac{\partial}{\partial x}+y\frac{\partial}{\partial z}, \  \mathring{e}_2(p):=\frac{\partial}{\partial y}-x\frac{\partial}{\partial z}, \ T(p):=\frac{\partial}{\partial z}.
\end{align}
The standard contact form is defined by $\Theta=dz+xdy-ydx$ and the contact plane $\xi_p$ at any point $p\in \mathbb{H}_1$ (or called horizontal distribution) is the smooth plane distribution generated by $\mathring{e}_1(p)$ and $\mathring{e}_2(p)$, equivalently, $\xi_p=ker\Theta$. We shall consider the (left invariant) Levi-metric $\langle\cdot, \cdot\rangle:=\frac{1}{2}d\Theta(\cdot, J\cdot)$ in $\mathbb{H}_1$ so that $\{\mathring{e}_1,\mathring{e}_2,T\}$ is an orthonormal basis in the Lie algebra of $\mathbb{H}_1$. The standard CR structure is an endomorphism $J:\mathbb{H}_1\rightarrow \mathbb{H}_1$ such that $J(\mathring{e}_1)=\mathring{e}_2$, $J(\mathring{e}_2)=-\mathring{e}_1$, $J(T)=0$, and $J^2=-1$.

Recall that the velocity vector $\gamma'$ of a curve $\gamma$ defined on some interval $I$ in $\mathbb{H}_1$ has the natural decomposition
\begin{align*}
\gamma'=\gamma'_\xi+\gamma'_T,
\end{align*}
where $\gamma'_\xi$ (resp. $\gamma'_T$) is the orthogonal projection of $\gamma'$ on $\xi$ along $T$ (resp. on $T$ along $\xi$) with respect to the Levi-metric. A \textit{horizontally regular curve} is a parametrized curve $\gamma(u)$ such that $\gamma'_\xi(u)\neq 0$ for all $u\in I $ (Definition 1.1, \cite{CHL}). In Proposition 4.1 \cite{CHL}, we showed that any horizontally regular curve can be uniquely reparametrized by horizontal arc-length $s$, up to a constant, such that $|\gamma'_\xi(s)|=1$ for all $s$, and called the curve being with horizontal unit-speed. Also, a curve $\gamma: I \subset \mathbb{R}\rightarrow \mathbb{H}_1$ is called \textit{horizontal} (or \textit{Legendrian}) if its tangent at any point on the curve is on the contact plane, namely, if we write the curve in coordinates $\gamma:=(x,y,z)$ with the tangent vector $\gamma'=(x',y',z')=x'\mathring{e}_1(\gamma)+y'\mathring{e}_2(\gamma)+T(z'-x'y+xy')$, then the curve $\gamma$ is horizontal if and only if
\begin{align}\label{horizontal}
z'-x'y+xy'=0,
\end{align}where the prime $'$ denotes the derivative with respect to the parameter of the curve.

Moreover, two geometric quantities for horizontally regular curves parametrized by horizontal arc-length, the \textit{p-curvature} $\kappa(s)$ and the \textit{contact normality} $\tau(s)$ (also called \textit{T-variation} in \cite{chiu2019global}), are defined respectively by
\begin{align}
\kappa(s)&:=\langle\frac{d\gamma'(s)}{ds},J\gamma'(s)\rangle,  \label{pcur} \\
\tau(s)&:=\langle\gamma'(s),T\rangle. \label{tau}
\end{align}
Note that both $\kappa$ and $\tau$ are invariant under pseudo-hermitian transformations $PHS(1)$ \cite[Section 4]{CHL}. We point out that the p-curvature $\kappa(s)$ is analogous to the curvature of the curve in the Euclidean space $\mathbb{R}^3$, while $\tau(s)$ measures how far the curve is from being horizontal. By \eqref{horizontal}, a curve is horizontal if and only if $\tau\equiv 0$. When the curve $\gamma(u)$ is parametrized by arbitrary parameter $u$ (not necessarily the horizontal arc-length $s$), the p-curvature and the contact normality are given by
\begin{align}\label{pcurve}
\kappa(u):=\frac{x'y''-x''y'}{\left((x')^2+(y')^2\right)^{3/2}}(u),
\end{align}
and
\begin{align}\label{contactnor}
\tau(u):=\frac{xy'-x'y +z'}{((x')^2+(y')^2)^{1/2}}(u),
\end{align}
respectively, and it is also clear that the p-curvature is the usual curvature of the projection of the curve onto the $xy$-plane. We also recall that \cite[p 12]{CHL} the Frenet frame formula for any horizontally regular curve $\gamma(s)$ parametrized by horizontal arc-length is given by
\begin{align}\label{Frenet}
\left\{
\begin{array}{ccccc}
\frac{d\gamma}{ds} &  = &  \gamma'_\xi  &  & +\tau T, \\
\frac{d\gamma'_\xi}{ds} &  = &                &\kappa J\gamma'_\xi, &\\
\frac{dJ\gamma'_\xi}{ds} & = &-\kappa \gamma'_\xi & &- T, \\
\frac{dT}{ds} & =  &0. &&
\end{array}
\right.
\end{align}

Finally we review some concepts of hypersurfaces in $\mathbb{H}_1$. Consider the hypersurface $\Sigma$ in $\mathbb{H}_1$. A point $p\in \Sigma$ is called \textit{singular} if $\xi_p$ coincides with the tangent plane $T_p\Sigma$ at $p$. Otherwise, $p$ is called \textit{nonsingular} or \textit{regular}, and so $T\Sigma\cap \xi$ defines a $1$-dimensional foliation (also called the characteristic direction). Choose a unit vector $e_1\in T\Sigma\cap \xi$, there is a unique (up to a sign) unit vector $e_2\in \xi$ that is perpendicular to $e_1$ with respect to the Levi-metric $\langle, \rangle$. In \cite{cheng2007} the authors call $e_2$ the \textit{p-normal} or the \textit{Legendrian normal}. Suppose that $\Sigma$ bounds a domain $M$ in $\mathbb{H}_1$. We define the p-area $2$-form $d\Sigma$ by computing the first variation, away from the singular set, of the standard volume in the p-normal $e_2$:
$$\delta_{fe_2}\Big( \int_M \Theta \wedge d\Theta\Big) =c\int_\Sigma f d\Sigma,$$
where $f\in C^\infty(\Sigma)$ with compact support away from the singular points, and $c=2$ is the normalization constant. The sign of $e_2$ is determined by requiring that $d\Sigma$ is positive w.r.t. the induced orientation on $\Sigma$. In \cite[(2.7), p261]{cheng2007} the authors derived the formula for p-areas for any hypersurface of graph-type in the Heisenberg groups $\mathbb{H}_n$ of higher dimensions $(n\geq 1)$. In particular, when $n=1$ and if $\Sigma=(x,y,f(x,y))$ for $(x,y)$ in some domain $D\subset \mathbb{R}^2$, the p-area of $\Sigma$ is given by
\begin{align}\label{pareagraph}
\mathcal{A}(\Sigma)=\iint_{(x,y)\in D}\Big((f_x-y)^2+(f_y+x)^2\Big)^{1/2} dxdy,
\end{align}
which will be mentioned again in Example \ref{ex2} in Section \ref{pareavolumeexammples}. Note that the p-area form can be continuously extended over the whole surface $\Sigma$ by having $d\Sigma$ vanished at the set of singular points. The similar notions of volumes and areas are also studied by \cite{capogna, ritore, ritore2006} when considering the $C^1$-surface $\Sigma$ enclosing a bounded set $M$. We point out that the p-area of $\Sigma$ also coincides with the $(2n + 1)$-dimensional spherical Hausdorff measure of
$\Sigma$ (see \cite{balogh2003,franchi}).

\section{The representations of p-area forms, volumes, and examples}\label{pareavolumeexammples}
In this section we show several important representations for volumes and p-areas for domains with boundaries. At first, for any surface defined by \eqref{maincurve} we show that the horizontal length of the curve $\mathcal{C}_s(t):=\mathbb{X}(s, t)$ only depends on the functions $f(t)$ and $g(t)$.

\begin{lemma}\label{horizontallengthdetermined}
Let $\Sigma$ be a parametrized surface defined by $\mathbb{X}(s,t)=\gamma(s)+f(t)U(s)+g(t)V(s)+h(t)T$ for $(s,t)\in [0,s_0]\times [0,t_0]$, where $\gamma$ is a horizontally regular curve with horizontal arc-length. Then for any fixed $s\in [0, s_0]$ the horizontal length of the curve $\mathcal{C}_s(t)=\mathbb{X}(s,t)$ is given by
$$length(\mathcal{C}_s)=\int_0^{t_0}\sqrt{(f')^2+(g')^2}dt.$$
Thus, the curve $\mathcal{C}_s(t)$ is parametrized by the horizontal arc-length if and only if $(f')^2 + (g')^2=1$ for all $t\in [0, t_0]$.
\end{lemma}
\begin{proof}
Since $\gamma(s)=\big(x(s), y(s), z(s)\big)$, we have $U=x'\mathring{e}_1+y' \mathring{e}_2$ and $V=-y'\mathring{e}_1+x'\mathring{e}_2$. By taking the derivative with respect to $t$ for any fixed $s$, one has that $\frac{d}{dt}\mathcal{C}_s=f' U+g'V+h'T=(f'x'-g'y')\mathring{e}_1+(f'y'+g'x')\mathring{e}_2+h'T$. Then the horizontal length of the curve $\mathcal{C}_s$
\begin{align*}
length(\mathcal{C}_s)&=\int_0^{t_0} |\frac{d}{dt}\mathcal{C}_s(t)| dt \\
&=\int_0^{t_0} \big[ (f'x'-g'y')^2+(f'y'+g'x')^2 \big]^{1/2}  dt \\
&=\int_0^{t_0}\big[\big((f')^2+(g')^2\big)\big((x')^2+(y')^2\big)\big]^{1/2} dt \\
&=\int_0^{t_0} \big((f')^2+(g')^2 \big)^{1/2} dt.
\end{align*}
Therefore, the horizontal arc-length of $\mathcal{C}_s$ is completely determined by the integral of the term $\sqrt{(f')^2+(g')^2}$ and the result follows.
\end{proof}

Next, we derive the formula of p-areas for parametrized surfaces in terms of the double integral of its parameters.

\begin{proposition}\label{pareaform1}
Given a parametrized surface $\Sigma: \mathbb{X}(s,t)=\{\big( X(s,t), Y(s,t), Z(s,t)\big), \forall (s,t)\in D\subset \mathbb{R}^2 \}$ in $\mathbb{H}_1$ for some domain $D\subset \mathbb{R}^2$, the p-area form $d\Sigma_p$ at the nonsingular point $p\in \Sigma$ can be represented in terms of the components of $\mathbb{X}$ and its derivatives pointwise, namely,
\begin{align}\label{pareaform}
d\Sigma_p = \sqrt{A^2+B^2}ds dt,
\end{align}
where
\begin{align}
A&:=Z_s X_t- X_s Z_t +X(X_tY_s-X_sY_t), \label{AB} \\
B&:=Z_s Y_t- Y_s Z_t+ Y(X_tY_s -X_sY_t), \nonumber
\end{align}
and the subscripts denote the partial derivatives w.r.t. the specific variables.
The p-area of the surface $\Sigma$ is given by
\begin{align}\label{parea}
\mathcal{A}{(\Sigma)}=\int_{p\in \Sigma}d\Sigma_p=\iint_{(s,t)\in D} \sqrt{A^2+B^2}dsdt.
\end{align}
\end{proposition}

\begin{proof}
By the definition of p-areas \eqref{parea0}, first we seek the vector $e_1(p)\in T_pM\cap \xi_p$. Taking the derivatives with respect to the variables $s$ and $t$ on $\mathbb{X}$ respectively, one has that
\begin{align*}
\mathbb{X}_s(s,t)=(X_s, Y_s, Z_s), \ \mathbb{X}_t(s,t)=(X_t,Y_t, Z_t),
\end{align*}
and then use the (Euclidean) cross product to get the surface normal vector
\begin{align}\label{N1}
N_1=\mathbb{X}_s \times \mathbb{X}_t=(Z_tY_s-Z_sY_t, Z_sX_t-Z_tX_s, X_sY_t-X_tY_s).
\end{align}
In addition, for any point $p=(X(s,t), Y(s,t), Z(s,t))\in \Sigma$, one can use the standard basis $\mathring{e}_1= (1,0, Y), \mathring{e}_2=(0, 1, -X)$ of the contact plane $\xi_p=span\{\mathring{e}_1(p), \mathring{e}_2(p)\}$ to have the vector $N_2$ perpendicular to the contact plane, namely,
\begin{align}\label{N2}
N_2=\mathring{e}_1 \times \mathring{e}_2 = (-Y, X, 1).
\end{align}
Since the vector along $T_pM \cap \xi_p$ is perpendicular to the normal vectors $N_1$ and $N_2$ of $T_pM$ and $\xi$, respectively, in the usual Euclidean metric, it must be parallel to the cross product of $N_1$ and $N_2$. Thus, by \eqref{N1}, \eqref{N2}, the vector $e_1\in T_pM\cap \xi_p$ is parallel to the vector
\begin{align*}
N_1 \times N_2 =
\Big(Z_s X_t- X_s Z_t +X(X_tY_s-X_sY_t)\Big)\mathring{e}_1 + \Big( Z_s Y_t- Z_t Y_s + Y(X_tY_s -X_sY_t) \Big)\mathring{e}_2.
\end{align*}
Set $A=Z_s X_t- X_s Z_t +X(X_tY_s-X_sY_t)$ and $B=Z_s Y_t- Z_t Y_s + Y(X_tY_s -X_sY_t)$, and we can choose the (horizontal) unit vector $e_1\in T_pM\cap \xi_p$ and $e_2= Je_1$ as
\begin{align}\label{e1e2}
\begin{array}{rl}
e_1&=\frac{A}{\sqrt{A^2+B^2}}\mathring{e}_1 + \frac{B}{\sqrt{A^2+B^2}}\mathring{e}_2 := (\cos\phi) \mathring{e}_1 + (\sin\phi) \mathring{e}_2, \\
e_2&=\frac{-B}{\sqrt{A^2+B^2}}\mathring{e}_1 + \frac{A}{\sqrt{A^2+B^2}}\mathring{e}_2:= (-\sin\phi) \mathring{e}_1 + (\cos\phi )\mathring{e}_2,
\end{array}
\end{align}
where $\phi$ is the angle between $\mathring{e}_1$ and $e_1$ following the orientation defined on the contact plane. Now we have the orthonormal frame $\{e_1, e_2, T \}$ w.r.t. the Levi-metric in $\mathbb{H}_1$. Since the standard orthonormal frame $\{\mathring{e}_1, \mathring{e}_2, T\}$ and its coframe $\{dx, dy, \Theta\}$ are dual to each other, by \eqref{e1e2}, the dual basis $\{e^1,e^2, \Theta\}$ of $\{e_1, e_2, T\}$ can be represented by the linear combination of the coframe $dx,dy, dz$, namely,
\begin{align}\label{duale1e2}
\begin{array}{rl}
e^1&= (\cos\phi) dx + (\sin\phi) dy, \\
e^2&= (-\sin\phi)dx+ (\cos\phi) dy, \\
\Theta&=dz+xdy-ydx.
\end{array}
\end{align}
When restrict the infinitesimal changes $dx, dy, dz$ on the surface $\Sigma$, one has that $dx=X_sds +X_tdt$, $dy = Y_sds+ Y_tdt$, $dz=Z_sds+Z_tdt$. Substitute $dx, dy, dz$ into \eqref{duale1e2} to have
\begin{align}\label{duale123}
\begin{array}{ccc}
e^1&= &( \cos\phi\  X_s + \sin\phi \ Y_s )ds + ( \cos\phi \ X_t + \sin\phi \ Y_t)dt, \\
e^2&= &( -\sin\phi\  X_s + \cos\phi \ Y_s )ds + ( -\sin\phi \ X_t + \cos\phi \ Y_t)dt, \\
\Theta&=& (Z_s+XY_s-YX_s)ds + (Z_t+XY_t -YX_t)dt.
\end{array}
\end{align}
Note that the p-area form is well-defined on the whole surface $\Sigma$ (up to a sign), i.e., making $d\Sigma$ vanished at singular points. Thus, by \eqref{e1e2} and \eqref{duale123} we have reached the p-area form
\begin{align}
\Theta\wedge e^1
&= \Big[-\cos\phi\big(Z_tX_s -Z_s X_t +X(X_s Y_t-X_tY_s)\big)- \sin\phi \big(Z_t Y_s-Z_s Y_t + Y(X_s Y_t-X_t Y_s)\big)\Big]ds\wedge dt \\ \nonumber
&= (A \cos\phi  + B\sin\phi ) ds\wedge dt\\ \nonumber
&=\sqrt{A^2+B^2}ds\wedge dt,  \nonumber
\end{align}
and the result follows.
\end{proof}

Usually the computations for getting the functions $A,B$ in Proposition \ref{pareaform1} are complicated. The next lemma provides an relatively efficient method to simply the computations. Recall that the cross product $\bar{\times}$ with respect to the standard basis $\{\mathring{e}_1, \mathring{e}_2, T\}$ defined on the Heisenberg group $\mathbb{H}_1$ satisfies
\begin{align}\label{Hcrossproduct}
\mathring{e}_1 \bar\times \mathring{e}_2 =T, \ \mathring{e}_2 \bar\times T= \mathring{e}_1, \ T \bar\times \mathring{e}_1 =\mathring{e}_2,
\end{align}and is extended linearly on the coefficients. For instance, if $u_i=a_i \mathring{e}_1 +b_i\mathring{e}_2 + c_i T$ for $i=1,2$, then $u_1\bar\times u_2= (b_1c_2-b_2c_1)\mathring{e}_1 + (a_2c_1-a_1c_2)\mathring{e}_2 + (a_1b_2-a_2b_1)T $.

\begin{lemma}\label{ABinner}
The coefficients $A, B$ defined by \eqref{AB} can be represented in terms of the inner product of the standard basis $\mathring{e}_1, \mathring{e}_2$, and the vector $\mathbb{X}_s\bar{\times} \mathbb{X}_t$, with respect of the Levi-metric, namely,
\begin{align}\label{innerAB}
A& =\langle \mathbb{X}_s \bar\times \mathbb{X}_t, \mathring{e}_2 \rangle, \\ \nonumber
B& =\langle \mathbb{X}_s \bar\times \mathbb{X}_t, -\mathring{e}_1 \rangle,
\end{align}where $\bar{\times}$ is the cross product in $\mathbb{H}_1$ defined by \eqref{Hcrossproduct}.
\end{lemma}
\begin{proof}
To simply the computation we ignore the variables $s, t$, and use the same notations as before. For any point $\mathbb{X}$ on the surface $\Sigma$, we have the representation
\begin{align*}
\mathbb{X}=(X,Y,Z)= X\mathring{e}_1 +Y\mathring{e}_2 + Z T
\end{align*}with the partial derivatives
\begin{align}
\mathbb{X}_s&=(X_s, Y_s, Z_s)=X_s \mathring{e}_1 + Y_s \mathring{e}_2 +(-YX_s +XY_s +Z_s)T, \\
\mathbb{X}_t&=(X_t, Y_t, Z_t)=X_t \mathring{e}_1 + Y_t \mathring{e}_2 +(-YX_t +XY_t +Z_t)T. \nonumber
\end{align}

Using \eqref{Hcrossproduct}, a straight forward computation shows that
\begin{align*}
\mathbb{X}_s\bar\times \mathbb{X}_t
& = \big(   Z_tY_s-Z_sY_t + Y(X_sY_t -X_tY_s )\big)\mathring{e}_1 + \big( Z_sX_t-Z_tX_s + X(X_tY_s - X_sY_t )\big)\mathring{e}_2 \\
&\hspace{1cm}+ (X_sY_t-X_tY_s) T \\
& \overset{\eqref{AB}}{=} -B \mathring{e}_1 +A \mathring{e}_2 +(X_sY_t-X_tY_s) T.
\end{align*}
In conclusion, the coefficients $A, B$ above can be obtained by using the inner product induced from the Levi-metric, namely,
\begin{align*}
A& =\langle \mathbb{X}_s \bar\times \mathbb{X}_t, \mathring{e}_2 \rangle, \\ \nonumber
B& =\langle \mathbb{X}_s \bar\times \mathbb{X}_t, -\mathring{e}_1 \rangle.
\end{align*}
\end{proof}

The next two examples show that the p-areas and the usual $\mathbb{R}^2$-areas coincide when the region is on the vertical planes, but not on the horizontal planes.
\begin{example}\label{ex1}
Consider the vertical plane (perpendicular to the $xy$-plane) such that its projection onto the $xy$-plane is a line with slope $a\in (-\infty, \infty)$. The vertical plane can be parametrized by $\mathbb{X}(s,t)=(s,as+b, t)$ for any $s,t, b\in \mathbb{R}$. Let $\mathcal{R}$ be a plane region on the vertical plane with boundary, then $\mathcal{R}$ can be given by $\mathbb{X}(s,t)=(s,as+b, t)$ for $(s,t)$ in some domain $D\subset \mathbb{R}^2$. Set $X(s,t)=s, Y(s,t)=as+b, Z(s,t)=t$ and use \eqref{AB}\eqref{pareaform}, one gets $A=-1$, $B=-a$, and
$$\mathcal{A}(\mathcal{R})=\iint_D \sqrt{1+a^2} dsdt=\iint_D |\mathbb{X}_s \times \mathbb{X}_t| dsdt =\mathcal{A}_{\mathbb{R}^2}(\mathcal{R}),$$
where $\mathcal{A}_{\mathbb{R}^2}(\mathcal{R})$ is the usual Euclidean area of $\mathcal{R}$. Similarly, the p-area and the $\mathbb{R}^2$-area of any region $\mathcal{R}$ coincide if $\mathcal{R}$ is on the $yz$-plane.
\end{example}

\begin{example}\label{ex2}
Consider the graph $\mathbb{X}(s,t)=(s,t,f(s,t))$ for some function $f(s,t)$ defined on $D\subset \mathbb{R}^2$ and denote its components by $X(s,t)=s, Y(s,t)=s, Z(s,t)=f(s,t)$. It is an immediate result that the p-area for the graph is given by $\mathcal{A}(\mathbb{X})=\iint_D\sqrt{(s-f_t)^2 + (t+f_s)^2}ds dt$ by \eqref{pareaform}, which coincides with the formula of p-areas for graphs derived in \cite[(2.7), p261]{cheng2007} or see \eqref{pareagraph}. When $f\equiv const.$, the $\mathcal{A}(\mathbb{X})=\iint_D\sqrt{s^2+t^2}dsdt$, which is \textit{not} equal to the $\mathbb{R}^2$-area of the graph.
\end{example}

One has an immediate application of Example \ref{ex1}. The following result shows that Pappus-Guldin Theorem for volumes holds in $\mathbb{H}_1$ when the domain is obtained by parallel moving a plane region along a line on the $xy$-plane.

\begin{corollary}
Let $\gamma(s)$ be a line segment with horizontal length $s_0$ on the $xy$-plane in $\mathbb{H}_1$ and $P$ a plane perpendicular to $\gamma(s)$. Suppose $\mathcal{R}\subset P$ is a plane region. Then the volume of the solid cylindrical domain $M=\mathcal{R}\times \gamma$ obtained by parallel moving $\mathcal{R}$ along $\gamma$ satisfies Pappus-Guldin Theorem. By parallel moving we mean that when the region $\mathcal{R}$ moves along the line segment $\gamma$, $\mathcal{R}$ is always perpendicular to $\gamma$ in the Euclidean sense.
\end{corollary}

\begin{proof}
If $\gamma, \mathcal{R}\subset P$, and $M$ are defined as the assumptions, Pappus-Guldin Theorem for volumes in the Euclidean spaces (Theorem A) implies that
$$Volume(M)= \mathcal{A}_{\mathbb{R}^2}(\mathcal{R})\cdot L,$$
where $L$ is the Euclidean length of $\gamma$. Additionally, by the construction and Example \ref{ex1}, we have $\mathcal{A}(\mathcal{R})=\mathcal{A}_{\mathbb{R}^2}(\mathcal{R})$; moreover, any curve contained in the $xy$-plane has the same horizontal and Euclidean lengths, namely, $s_0=L$. Therefore, $Volume(M)=\mathcal{A}(\mathcal{R})\cdot s_0$ and we complete the proof.
\end{proof}

Given a domain $M\subset \mathbb{H}_1$ with boundary $\partial M$, and let $\{ e_1, e_2, T\}$ be an orthonormal frame w.r.t. Levi-metric in $\mathbb{H}_1$ such that $e_1\in TM\cap \xi$ and $e_2=Je_1$. Denote by $\{e^1, e^2, \Theta\}$ the corresponding coframe of $\{e_1, e_2, T\}$. Next we show the divergence theorem, which is inspired by \cite[Lemma 3.2]{CCW}.

\begin{proposition}[Divergence Theorem]\label{Divergence}
Given any domain $M$ with boundary $\partial M$ in the 3-dimensional Heisenberg group $\mathbb{H}_1$. For any vector field $P=P_1 e_1 +P_2 e_2+P_3 T$ defined on the tubular neighborhood of $M$ for any $C^1$-functions $P_1, P_2, P_3$, we have
\begin{align*}
\iiint_M div_b P d\mu =\iint_{\partial M} P_2d\Sigma_p,
\end{align*}where $div_b$ is the subdivergence, the volume form $d\mu = \Theta \wedge e^1 \wedge e^2$ and the p-area form $d\Sigma_p=\Theta \wedge e^1$ for $e_1\in TM\cap \xi$ and $e_2=Je_1$.
\end{proposition}

\begin{proof}
Recall that (\cite[p175]{CHMY})
\begin{align*}
\nabla e_1 &=\omega \otimes e_2, \nabla e_2=-\omega \otimes e_1, \nabla T\equiv 0, \\
de^2 &=-\omega \wedge e^1 =-\omega(e_2)e^2\wedge e^1 \mod(\Theta), \\
de_1 &=\omega \wedge e^2  =\omega(e_1)e^1\wedge e^2 \mod(\Theta), \\
d\Theta&=2e^1\wedge e^2,
\end{align*}
where $\nabla$ is the pseudo-hermitian connection, $\omega$ is the connection $1$-form, and $e^1, e^2$ are the dual $1$-form of the vectors $e_1, e_2$, respectively. Thus, on one hand, taking the exterior derivatives
\begin{align}\label{div1}
d(-P_1\Theta\wedge e^2 &+P_2 \Theta\wedge e^1) \\
&=-e_1(P_1)e^1\wedge \Theta \wedge e^2 -P_1 d(\Theta \wedge e^2) +e_2(P_2)e^2\wedge \Theta \wedge e^1 +P_2 d(\Theta\wedge e^1) \nonumber \\
&=e_1(P_1)\Theta\wedge e^1\wedge e^2 +P_1\Theta \wedge de^2 +e_2(P_2)\Theta \wedge e^1\wedge e^2 -P_2\Theta\wedge de^1 \nonumber \\
&=\Big[ e_1(P_1)+e_2(P_2)+P_1\omega (e_2)-P_2\omega(e_1) \Big]\Theta \wedge e^1 \wedge e^2. \nonumber
\end{align}
On the other hand, using the orthogonality of $\{e_1, e_2, T\}$,
\begin{align}\label{div2}
div_b P :&=\langle \nabla_{e_1}P, e_1 \rangle +\langle \nabla_{e_2}P, e_2\rangle \\
&=\langle e_1(P_1)e_1+P_1\nabla_{e_1}e_1 +e_1(P_2)e_2 +P_2\nabla_{e_1}e_2+e_1(P_3)T+P_3\nabla_{e_1}T, e_1 \rangle \nonumber \\
&\hspace{1cm} +\langle e_2(P_1)e_1+P_1\nabla_{e_2}e_1 +e_2(P_2)e_2 +P_2\nabla_{e_2}e_2+e_2(P_3)T+P_3\nabla_{e_2}T , e_2 \rangle \nonumber \\
&=e_1(P_1)+e_2(P_2)+P_1\omega(e_2)-P_2\omega(e_1). \nonumber
\end{align}
Combine \eqref{div1}, \eqref{div2} we have
\begin{align}\label{div3}
div_b P\ \Theta\wedge e^1 \wedge e^2 = d(-P_1\Theta \wedge e^2 +P_2 \Theta \wedge e^1)
\end{align}
on $M$. Use \eqref{div3} and the Stoke's theorem,
\begin{align*}
\iiint_M div_b P d\mu &=\iiint_M d(-P_1\Theta \wedge e^2 +P_2 \Theta \wedge e^1) \\
&=\iint_{\partial M} -P_1 \Theta \wedge e^2 +P_2 \Theta \wedge e^1 \\
&=\iint_{\partial M} P_2 \Theta \wedge  e^1 \\
&=\iint_{\partial M} P_2 d\Sigma_p.
\end{align*}
Here we use the fact that $e_1\in TM\cap \xi$ implies that $\Theta \wedge e^2|_{\partial M}=0$ on the third identity above.
\end{proof}

By the divergence theorem, the volume of any domain $M\subset \mathbb{H}_1$ with boundary $\partial M$ can be represented by the boundary integral with the angle $\phi$ between $\mathring{e}_1$ and $e_1\in TM\cap \xi$. In the proof of next corollary, we also show that there is only one connection $1$-form $\omega=d\phi$, where $\phi$ is the angle between the standard unit vector $\mathring{e}_1$ and $e_1\in T\Sigma\cap\xi$.

\begin{corollary}\label{divcoro}
Let $M$ be a domain in the 3-dimensional Heisenberg group $\mathbb{H}_1$ with boundary $\partial M$. Then the volume of $M$ is given by
$$V(M)=\frac{1}{2}\iint_{p\in\partial M} (-x\sin\phi +y\cos\phi) d\Sigma_p,$$
where $p=(x,y,z)\in \partial M$ and $\phi=\phi(p)$ the angle between $\mathring{e}_1(p)$ and the characteristic unit vector $e_1(p)\in T_pM\cap \xi_p$. In addition, suppose the boundary $\partial M$ is described by the parametrized surface
$$\mathbb{X}(s,t)=\{\big( X(s,t), Y(s,t), Z(s,t)\big) | \forall (s,t)\in D\},$$
then we have the volume of $M$ (up to a sign)
\begin{align}\label{volume}
V(M)=\frac{1}{2}\iint_{(s,t)\in D} (-XB+YA) dsdt
\end{align}
where $A,B$ are defined by \eqref{AB}.
\end{corollary}

\begin{proof}
By setting $e_1=\cos \phi \mathring{e}_1 +\sin \phi \mathring{e}_2$ and $e_2= -\sin \phi \mathring{e}_1 + \cos \phi \mathring{e}_2$ for some angle $\phi$ between $\mathring{e}_1$ and $e_1$ at the nonsingular point on $\Sigma$, we can write the position vector $P$ at the point $p=(x,y,z)$ by
\begin{align}\label{positionvector}
P(p)&=(x,y,z) \\
&=x\mathring{e}_1(p)+y\mathring{e}_2(p)+zT \nonumber \\
&=\underbrace{(x\cos\phi +y\sin\phi)}_{:=P_1}e_1+\underbrace{(-x\sin\phi +y\cos\phi)}_{:=P_2}e_2+zT. \nonumber
\end{align}

We shall find $e_1(P_1)+e_2(P_2)+P_1 \omega(e_2)-P_2 \omega(e_1)$ by the following computations:

\begin{align}
\mathring{e}_1(P_1)&=(\frac{\partial}{\partial x}+y\frac{\partial}{\partial z})(x \cos \phi +y \sin \phi) \label{e1p1} \\
&=\cos \phi +(-x\sin \phi +y \cos \phi)\frac{\partial\phi}{\partial x}+(-yx\sin +y^2 \cos\phi)\frac{\partial \phi}{\partial z}. \nonumber  \\
\mathring{e}_1(P_2)&=(\frac{\partial}{\partial x}+y\frac{\partial}{\partial z})(-x \sin \phi +y\cos \phi) \label{e1p2} \\
&=-\sin\phi +(-x\cos \phi -y \sin \phi)\frac{\partial\phi}{\partial x}+(-xy\cos\phi -y^2 \sin\phi)\frac{\partial \phi}{\partial z}. \nonumber  \\
\mathring{e}_2(P_1)&=(\frac{\partial}{\partial x}-x\frac{\partial}{\partial z})(x \cos \phi +y \sin \phi) \label{e2p1} \\
&=\sin \phi +(-x\sin \phi +y \cos \phi)\frac{\partial\phi}{\partial y}+(-yx\cos\phi +x^2 \sin\phi)\frac{\partial \phi}{\partial z}. \nonumber  \\
\mathring{e}_2(P_2)&=(\frac{\partial}{\partial y}-x\frac{\partial}{\partial z})(-x \sin \phi +y\cos \phi) \label{e2p2}\\
&=\cos\phi +(-x\cos \phi -y \sin \phi)\frac{\partial\phi}{\partial y}+(xy\sin\phi +x^2 \cos\phi)\frac{\partial \phi}{\partial z}. \nonumber
\end{align}
By \eqref{e1p1}-\eqref{e2p2}, we have
\begin{align}\label{e1p1e2p2}
e_1(P_1)+e_2(P_2)&=\cos\phi \mathring{e}_1(P_1) +\sin\phi\mathring{e}_2(P_1)-\sin\phi\mathring{e}_1(P_2) +\cos\phi\mathring{e}_2(P_2) \\
&=2+y\frac{\partial \phi}{\partial x}-x\frac{\partial \phi}{\partial y}+(x^2+y^2)\frac{\partial \phi}{\partial z}. \nonumber
\end{align}

Next we recall the Maurer-Cartan form for $\mathbb{H}_1$ (see \cite[Sec. 3]{sharpe2000differential} or \cite[Sec. 2]{CHL} for more detail). Consider the transformation rule between the orthonormal frames $\{\mathring{e}_1, \mathring{e}_2, T\}$ and $\{e_1, e_2, T \}$ given by
\begin{align*}
\begin{pmatrix}
e_1 \\ e_2 \\ T
\end{pmatrix}=
\begin{pmatrix}
\cos\phi & \sin\phi & 0 \\
-\sin\phi & \cos\phi & 0 \\
0 & 0 & 1
\end{pmatrix}
\begin{pmatrix}
\mathring{e}_1 \\ \mathring{e}_2 \\ T
\end{pmatrix}:= g \cdot
\begin{pmatrix}
\mathring{e}_1 \\ \mathring{e}_2 \\ T
\end{pmatrix}.
\end{align*}
Then we can get the Maurer-Cartan form
\begin{align*}
g^{-1}dg &=
\begin{pmatrix}
\cos\phi & -\sin \phi & 0 \\
\sin\phi & \cos \phi & 0 \\
0 & 0 & 1
\end{pmatrix}\cdot
\begin{pmatrix}
-\sin\phi d\phi & \cos\phi d\phi & 0 \\
-cos\phi d\phi & -\sin\phi d\phi & 0 \\
0 & 0 &
\end{pmatrix}\\
&=
\begin{pmatrix}
0 & d\phi & 0 \\
-d\phi & 0 & 0 \\
0 & 0 & 0
\end{pmatrix}:=
\begin{pmatrix}
\omega_1^1 & \omega_1^2 & \omega_1^3 \\
\omega_2^1 & \omega_2^2 & \omega_2^3 \\
\omega_3^1 & \omega_3^2 & \omega_3^3
\end{pmatrix}
\end{align*}
and the connection $1$-form $\omega:=\omega_1^2 =d\phi$.
\begin{align}
P_2\omega(e_1)&=P_2 d\phi(e_1)=P_2 e_1(\phi) \label{p2e1} \\
&=(-x \sin\phi +y\cos\phi)(\cos\phi \mathring{e}_1(\phi)+\sin\phi \mathring{e}_2(\phi))  \nonumber \\
&=(-x\sin \phi \cos\phi +y\cos^2\phi )\mathring{e}_1(\phi) +(-x\sin^2\phi+y\sin\phi\cos\phi)\mathring{e}_2(\phi), \nonumber \\
P_1\omega(e_2)&=P_1 d\phi(e_2)=P_1 e_2(\phi) \label{p1e2} \\
&=(x \cos\phi +y\sin\phi)(-\sin\phi \mathring{e}_1(\phi)+\cos\phi \mathring{e}_2(\phi))  \nonumber  \\
&=(-x\sin\phi\cos\phi -y\sin^2\phi)\mathring{e}_1(\phi)+(x\cos^2\phi+y\sin\phi\cos\phi)\mathring{e}_2(\phi). \nonumber
\end{align}
Combine \eqref{e1p1e2p2}, \eqref{p2e1}, and \eqref{p1e2} one gets
\begin{align}\label{div=2}
e_1(P_1)+e_2(P_2)-P_2\omega(e_1)+P_1\omega(e_2)=2
\end{align}

Apply Proposition \ref{Divergence} to the position vector $P=P_1e_1+P_2e_2+P_3T$ defined in \eqref{positionvector} on $M$ and use \eqref{div2}, \eqref{div=2}
\begin{align*}
div_b P &=\langle \nabla_{e_1}P, e_1 \rangle +\langle \nabla_{e_2}P, e_2\rangle  \\
&=e_1(P_1)+e_2(P_2)+P_1\omega(e_2)-P_2\omega(e_1)=2,
\end{align*}and
finally we have
\begin{align*}
2 V(M)&=2 \iiint_M d\mu \\
&=\iiint_M div_b P d\mu \\
&=\iint_{\partial M} P_2 d\Sigma_p \\
&=\iint_{\partial M}(\frac{-XB}{\sqrt{A^2+B^2}}+\frac{YA}{\sqrt{A^2+B^2}})\sqrt{A^2+B^2}dsdt \\
&=\iint_{\partial M} (-XB+YA) ds dt.
\end{align*}
Here we have used \eqref{e1e2} in Proposition \ref{pareaform1} for the fourth identity above and the result follows.
\end{proof}

\begin{remark}
A geometric interpretation of \eqref{volume} can be represented as follows: in $\mathbb{H}_1$, since the unit p-normal vector $e_2=\frac{-B}{\sqrt{A^2+B^2}}\mathring{e}_1+\frac{A}{\sqrt{A^2+B^2}}\mathring{e}_2$ by \eqref{e1e2} and the position vector $P|_{\partial M}=(X,Y,Z)$, the integrand on the right-hand side of \eqref{volume} actually is the inner product
$$e_2\cdot P \ d\Sigma_p=-XB+YA \ dsdt$$
w.r.t. the Levi-metric. Similar to the result in the Euclidean space $\mathbb{R}^3$, when applying the divergence theorem to the position vector $P$ of a closed parametrized surface $\mathbb{X}(s,t)=\big(X(s,t),Y(s,t),Z(s,t)\big)$ for $(s,t)\in D$, the volume of the enclosed domain $M$ is given by
$$Volume(M)=\frac{1}{3}\iiint_M div P d\mu =\frac{1}{3}\iint_{\partial M} P\cdot nd\sigma, $$
where $n$ is the outward unit vector and $d\sigma$ is the area element.
\end{remark}

\begin{example}
A torus $M$ in $\mathbb{H}_1$ is given by rotating a circle with radius $r$ centered at a bigger circle which is centered at the origin with radius $R$. More precisely, for any $0<r<R$, the parametrzied surface of a torus is represented by $\partial M: \mathbb{X}(s,t)=\big(X(s,t), Y(s,t), Z(s,t) \big)$, where
\begin{align*}
\left\{
\begin{array}{rl}
X(s,t)&= (R + r \cos s )\cos t , \\
Y(s,t)&= (R + r \cos s)\sin t, \\
Z(s,t)&= -r \sin s,
\end{array}
\right.
\end{align*}for $0\leq s,t\leq 2\pi$, $r<R$. By \eqref{AB}, we have $A=r(R+r \cos s)\big(\cos s \sin t + \sin s\cos t(R+r\cos s )\big)$ and $B=r(R+r\cos s )\big(-\cos s\cos t+\sin s \sin t (R+r\cos s ) \big)$. By \eqref{volume}, the volume of the torus is
\begin{align*}
V(M)&=\frac{1}{2}\int_0^{2\pi}\int_0^{2\pi} (-XB+YA) ds dt = \frac{1}{2}\int_0^{2\pi}\int_0^{2\pi} r \cos s (R+r\cos s)^2 ds dt =2 \pi^2 r^2 R.
\end{align*}
We can also obtain the (Euclidean) volume by Pappus-Guldin theorem, that is,
\begin{align*}
V(M)&=(\text{area of smaller disk})\times (\text{length of the center travelling around})\\
&=(\pi r^2)\cdot (2\pi R)=2\pi^2r^2 R.
\end{align*}
It shows that the volumes from different approaches coincide with each other.
\end{example}

An immediate application of Corollary \ref{divcoro} is that an isoperimetric-type inequality for the domain $M$ with boundary $\partial M$ can be obtained. Indeed, by applying the Cauchy–Schwarz inequality:  \begin{align}\label{cauchy}
-XB+AY\leq \sqrt{X^2+Y^2}\sqrt{A^2+B^2}
\end{align}
and using the Proposition \ref{pareaform1}, we have
$$\frac{V(M)}{\mathcal{A}(\partial M)}\leq \frac{1}{2}\Big( \sup_{{(x,y,z)\in \partial M}}\sqrt{x^2+y^2}\Big).$$
Note that the left-hand side above is invariant under pseudo-hermitian transformations. Denote $PSH_0(1)=\{g\in PSH(1)\ | gM \ni (0,0,0)\}$ by the subset of the group of pseudo-hermitian transformations $PSH(1)$ such that the origin $(0,0,0)$ is contained in $gM$ for any $g\in PSH(1)$ acting on $M$. Taking the infimum over all elements in $PSH_0 (1)$ gives us the lower bound of the maximal distance to the $z$-axis. We summarize all discussions above as the following corollary.

\begin{corollary}\label{isoperi}
For any domain $M$ in $\mathbb{H}_1$ with boundary $\partial M$. There holds an isoperimetric type inequality
\begin{align}\label{isoperi2}
V(M)\leq \frac{1}{2}\Big( \sup_{{(x,y,z)\in \partial M}}\sqrt{x^2+y^2} \Big) \cdot \mathcal{A}(\partial M).
\end{align}
In addition, denote $PSH_0(1)=\{g\in PSH(1)\ | gM \ni (0,0,0)\}$ by the subset of the group $PSH(1)$ of the pseudo-hermitian transformations, then we have
$$\frac{V(M)}{\mathcal{A}(\partial M)}\leq \frac{1}{2}\inf_{g\in PHS_0(1)}\Big(  \sup_{{(x,y,z)\in \partial gM}}\sqrt{x^2+y^2}  \Big).$$
\end{corollary}

\begin{remark}
The author point out that the result in Corollary \ref{isoperi} is unsatisfied since the inequality \eqref{isoperi2} is too rough to achieve the sharp case. Actually, it can be showed that when the equality holds for \eqref{isoperi2}, equivalently the identity holds for Cauchy-Schwarz inequality \eqref{cauchy}, the domain $M$ is a cylinder and the enclosed domain is an empty set, which violates the assumption of $M$.
\end{remark}

\section{The proof of Theorem \ref{mainvolume} and Corollary \ref{mainvolume2}}\label{mainvolumemainvolum2}
\begin{proof}[\textbf{Proof of Theorem \ref{mainvolume}}]
Let
\begin{align}\label{surface}
\mathbb{X}(s,t)=\gamma + fU +g V+g T
\end{align} be the parameterized surface for $(s,t)\in [0,s_0]\times [0,t_0]$ such that $\gamma$ is a horizontally regular curve with horizontal arc-length as before. Take the derivatives w.r.t. the parameters $s$ and $t$ respectively and use the Frenet frame formula \eqref{Frenet} to have
\begin{align}
\mathbb{X}_s&=\gamma' + f U'+gV' \label{xsformula}\\
&=(U +\tau T) + f(\kappa V) +g(-\kappa U -T) \nonumber \\
&=(1-\kappa g)U+\kappa f V+(\tau -g)T, \nonumber \\
\mathbb{X}_t&=f' U+ g'V +h'T. \label{xtformula}
\end{align}
Since $U=x'\mathring{e}_1+y'\mathring{e}_2$ and $V=-y'\mathring{e}_1+x'\mathring{e}_2$, by the cross product \eqref{Hcrossproduct} we can have
\begin{align}
\mathbb{X}_s\bar\times \mathbb{X}_t &=\big(\kappa h'f -g'(\tau-g)\big) U+\big(-h'(1-\kappa g)+f'(\tau -g)\big)V +\big((1-\kappa g)g' -\kappa f f'\big)T  \label{xstimesxt} \\
&=[\big(\kappa h' f -g'(\tau -g)\big) x'+\big(h'(1-\kappa g)-f'(\tau -g) \big)y']\mathring{e}_1  \nonumber \\
&\hspace{2cm}+ [\big( \kappa h'f-g'(\tau -g) \big)y' +\big( -h'(1-\kappa g)+f'(\tau-g) \big)x']\mathring{e}_2 \nonumber \\
&\hspace{2.5cm} + \big((1-\kappa g)g' -\kappa f f'\big)T.
\end{align}
By Lemma \ref{ABinner}, one gets
\begin{align}
A&=\big( -h'(1-\kappa g)+f'(\tau-g) \big)x'+\big( \kappa h'f-g'(\tau -g) \big)y':=\alpha x' +\beta y',  \label{volumeAB} \\
B&=\big(-\kappa h' f +g'(\tau -g)\big) x'+\big(-h'(1-\kappa g)+f'(\tau -g) \big)y':=-\beta x' +\alpha y'.\nonumber
\end{align}
To obtain the volume of $M$, we shall use Corollary \ref{divcoro}. By \eqref{surface}, we write the components $X,Y$ of the surface $\mathbb{X}=(X,Y,Z)$ in terms of $\gamma$, $\gamma', f$ and $g$, namely, $X=x+fx'-gy'$ and $Y=y+fy'+gx'$. Substitute $X,Y$ into \eqref{volume} and use \eqref{volumeAB} to have
\begin{align*}
-XB+YA&=\beta (xx'+yy')+\alpha (yx'-xy')+\beta f +\alpha g.
\end{align*}
Here we have used the arc-length condition for $\gamma$: $(x')^2+(y')^2=1$. Expand the right-hand side  above then we have reached
\begin{align}\label{uglyform}
2 Volume(M)&=\int_0^{s_0}\int_0^{t_0}  \Big(\beta (xx'+yy')+\alpha (yx'-xy')+\beta f +\alpha g\Big)  dsdt\\
&=\iint \kappa (xx'+yy')h'f -\iint(xx'+yy')\tau g' -(xx'+yy')gg'-\iint (yx'-y'x)h'\nonumber  \\
&+\iint (yx'-y'x)\kappa h'g +\iint (yx'-y'x)\tau f' -\iint (yx'-y'x)f'g +\iint \kappa h'f^2 \nonumber \\
&-\iint \tau g'f +\iint gg'f -\iint h'g +\iint \kappa g^2 h' +\iint\tau f'g -\iint f'g^2 \nonumber  \\
&=I_1 +\cdots +I_{13}, \nonumber
\end{align}
and each double integral (here and hereafter) is taken over the domain $[0, s_0]\times [0, t_0]$. Note that each of these integrals contains terms of the form $H(s)J(t)$, and each such case we can write
\begin{align}\label{HJ}
\iint H(s)J(t) dsdt = \int_0^{s_0}\int_0^{t_0} H(s)J(t)dsdt = \Big(\int_0^{s_0} H(s)ds\Big)\Big(\int_0^{t_0} J(t)dt\Big).
\end{align}
Furthermore it frequently happens the factor on the right side of \eqref{HJ} is zero, because of the closedness and smoothness of the curve $\mathcal{C}_s$. Thus, we have $I_2, I_3$, and $I_5$ in \eqref{uglyform} are all zero. Moreover, since $\tau=z'-x'y+xy'$ by definition, $I_6+I_{12}=\iint (-yx'+y'x+\tau)f'g=\iint (2\tau -z')f'g$ and we complete the proof.
\end{proof}

\begin{proof}[\textbf{Proof of Corollary \ref{mainvolume2}}]
When the p-curvature $\kappa\equiv 0$ and $f\equiv 0$, by \eqref{uglyform} the volume of $M$ is given by
\begin{align}\label{uglyform2}
Volume(M)=-\frac{1}{2}\int_0^{s_0} \int_0^{t_0} h'g dsdt =-\frac{s_0}{2} \int_0^{t_0} h'g dt.
\end{align}
By assumption $f\equiv 0$, the solid of revolution $\mathbb{X}(s,t)=\gamma +gV+hT$ is obtained by moving the plane curve $gV+hT$ on the vertical plane (perpendicular to the $xy$-plane) along the curve $\gamma$. We observe that any vertical plane $V_p$ (perpendicular to the $xy$-plane) spanned by the vectors $V(s)$ and $T$ at the point $p=\gamma(s)$ can be obtained by moving some vertical plane $V_\mathcal{O}$ spanned by $e_2(\mathcal{O})$ and $T$ at the origin $\mathcal{O}$ to the point $p$ via the left translation $L_{\gamma(s)}$, namely, $V_p:=span_{p} \{V(s), T\}=L_{p} V_\mathcal{O}$. As shown in Example \ref{ex1}, any plane region $\mathcal{R}$ on the vertical plane has the same p-area $\mathcal{A}$ and $\mathbb{R}^2$-area $\mathcal{A}_{\mathbb{R}^2}$. Therefore, by the invariant property of p-areas and $\mathbb{R}^2$-areas, we have
$$\frac{\mathcal{A}(\mathcal{R}\subset V_p)}{\mathcal{A}_{\mathbb{R}^2}(\mathcal{R}\subset V_p)}=\frac{\mathcal{A}(\mathcal{R}\subset V_0)}{\mathcal{A}_{\mathbb{R}^2}(\mathcal{R}\subset V_0)}=1,$$
for any point $p\in \gamma$. Thus, using Green's theorem, the right-hand side of \eqref{uglyform2} becomes to

$$-\frac{s_0}{2} \int_0^{t_0} h'g dt=-\frac{s_0}{2} \oint_{\partial \mathcal{R}} g dh
=\frac{s_0}{2} \iint_\mathcal{R} dYdZ
=\frac{s_0}{2} \cdot \mathcal{A}_{\mathbb{R}^2}(\mathcal{R})
=\frac{s_0}{2} \cdot\mathcal{A}(\mathcal{R}),$$
where the orientation of the boundary $\partial \mathcal{R}$ is chosen such that the p-area is positive (the second identity above).
\end{proof}

\begin{remark}We interpret the geometric assumptions for  $\kappa\equiv 0$ and $f(t)\equiv 0$ as follows: according to \eqref{pcurve}, the p-curvature of any horizontally regular curve $\gamma$ is the usual (Euclidean) curvature of the projection of $\gamma$ onto the $xy$-plane, and so by $\kappa\equiv 0$ we mean the projection of $\gamma$ onto the $xy$-plane is a line. Besides, it is also natural to assume that $f\equiv 0$ since the tangential contribution of the tube-shaped domain has been considered by the curve $\gamma$ and hence one shall ignore the other tangential influence made by the vector $U$.
\end{remark}

\section{The proof of Theorem \ref{mainarea} and Corollary \ref{pappsarea}}\label{mainareapappsarea}
Before proving Theorem \ref{mainarea}, we show some observations when we started trying to prove the analogous Pappus-Guldin theorem in $\mathbb{H}_1$. We will show that any cylindrical surface, whose generatrix is along the $z$-axis and its directrix is a closed plane curve perpendicular to the $z$-axis, satisfies Pappus-Guldin theorem for areas. We point out that the horizontal length of any plane curve $\gamma(s)=\big( f(s), g(s), t\big)$ on the plane $\{t=const.\}$ is exactly same as the usual Euclidean length of $\gamma(s)$.

\begin{proposition}\label{observ1}
Let $\Sigma: \mathbb{X}(s,t)=\big(R(t)\cos(s), R(t)\sin(s), h(t)\big)$ be a surface of revolution in the Heisenberg group $\mathbb{H}_1$ generated by rotating the curve $\{(R(t), 0, h(t))\}$ about the $z$-axis for $s\in [0, 2\pi]$ and $t\in [0,t_0]$. Then the p-area of $\Sigma$ is given by
\begin{align*}
\mathcal{A}(\Sigma)=\int_0^{t_0} \int_0^{2\pi} |R(t)|\Big( \big(h'(t)\big)^2+\big(R(t)R'(t)\big)^2\Big)^{1/2} dsdt.
\end{align*}
In particular, when $R(t)=R$, a nonzero positive constant, and $h(t)=t$, we obtain the p-area for the standard cylinder
\begin{align*}
\mathcal{A}(\Sigma)&=2\pi R \cdot  t_0\\
&=(perimeter \ of \ the \ cross \ section)\cdot (length \ of\  the\  axis),
\end{align*}which satisfies Pappus-Guldin theorem.
\end{proposition}

\begin{proof}
Denote by the components $X(s,t)=R(t)\cos(s), Y(s,t)=R(t)\sin(s), Z(s,t)=h(t)$ of the parametrized surface $\mathbb{X}(s,t)$ and by \eqref{AB} we calculate $A=R(t)h(t)' \sin(s)+R(t)^2R(t)'\cos(s)$ and $B=-Rh'\cos(s)+R^2R'\sin(s)$. Using Proposition \ref{pareaform1} to have the first result. In addition, by the assumptions $R(t)=R$ and $h(t)=t$, the second result follows immediately.
\end{proof}

In general, Pappus-Guldin theorem holds for all cylindrical surfaces with the rotating axis parallel to the $z$-axis.
\begin{proposition}\label{observ2}
Any cylindrical surface with the axis parallel to the $z$-axis in the Heisenberg group $\mathbb{H}_1$ satisfies Pappus-Guldin theorem. The cylindrical surface defined by
$$\Sigma: \mathbb{X}(s,t)= \{\big(f(s), g(s), t\big)| \forall (s,t)\in [0, 2\pi]\times [0,t_0]\},$$ for some real numbers $t_0>0$ and $C^1$-functions $f(t),g(t)$, has the p-area
$$\mathcal{A}(\Sigma)=(\text{perimeter of the closed plane curve generated by } f \text{ and } g)\cdot t_0.$$
\end{proposition}

\begin{proof}
By \eqref{AB}, a simple calculation shows that $A=f'(t)$ and $B=g'(t)$. Thus, one has that the p-area $\mathcal{A}(\Sigma)=\int_0^{t_0} \int_0^{2\pi} \sqrt{\big(f'(s)\big)^2 +\big(g'(s)\big)^2} dsdt$ by \eqref{parea}, and the result follows.
\end{proof}

Now we derive a formula for p-area forms for the parametrized surfaces in $\mathbb{H}_1$.

\begin{proof}[\textbf{Proof of Theorem \ref{mainarea}}]
The proof actually has been shown in the content of Theorem \ref{mainvolume}. Let $\Sigma: \mathbb{X}(s,t)=\gamma+fU+gV+hT$ be the parametrized surface as the previous sections. By the calculations for $\mathbb{X}_s, \mathbb{X}_t$, and $\mathbb{X}_s\bar\times \mathbb{X}_t$ (use \eqref{xsformula}, \eqref{xtformula}, \eqref{xstimesxt}) one gets the functions $A$ and $B$ as shown in \eqref{volumeAB}. Therefore,
\begin{align*}
A^2+B^2&=\alpha^2+\beta^2\\
&=(h')^2\big((\kappa f)^2 +(1-\kappa g)^2 \big)-2h'(\tau-g)\big(\kappa(fg'-f'g)+f'\big)+(\tau-g)^2\big( (f')^2+(g')^2\big),
\end{align*} where we have used the assumption of arc-length parameter for the curve $\gamma(s)$: $(x')^2+(y')^2=1$, and we complete the proof of the theorem.
\end{proof}

\begin{proof}[\textbf{Proof of Corollary \ref{pappsarea}}]
Substitute the function $h(t)$ by $h(t)\equiv 0$ into Theorem \ref{mainarea} with the assumption $(f')^2 +(g')^2=1$, we immediately have the result \eqref{pappasarea1}. It is also clear that the second result \eqref{pappasarea2} follows by the assumption.
\end{proof}

Two examples satisfying two conditions (1) and (2) in Corollary \ref{pappsarea} are provided respectively.

\begin{example}\label{examplereqornotone}
Let $\gamma(s)=\big( R \sin(\frac{s}{R}), -R \cos(\frac{s}{R}), (1-R)s \big)$ for any $R>0$, $s\in[0, 2\pi]$. Use \eqref{pcurve}, \eqref{contactnor}, we have $U(s)=(\cos(\frac{s}{R}), \sin(\frac{s}{R}), -R)$, $V(s)=(-\sin(\frac{s}{R}), \cos(\frac{s}{R}), 0)$, $\kappa\equiv \frac{1}{R}$, and $\tau\equiv 1$. Let $f(t)=a \sin(t)$, $g(t)=b \cos(t)$ for any $a, b\in \mathbb{R}\setminus\{0\}$, $t\in [0, 2\pi]$, $|a|>1$, $|b|\leq 1$ to avoid self-intersected. Then the parametrized surface defined by $\Sigma: \mathbb{X}(s,t)=\gamma(s)+f(t)U(s)+g(t)V(s)$ satisfies the assumption of Corollary \ref{pappsarea} (1) with the p-area $\mathcal{A}(\Sigma)=4 \pi^2$. Note that
\begin{enumerate}
\item  when $R=1$, the curve $\gamma$ is a circle on the $xy$-plane centered at the origin. The surface $\Sigma$ is a surface of revolution obtained by rotating an ellipse $\mathcal{C}_s=\big(f(t), -R+g(t), -Rf(t)\big)$ along the $z$-axis such that the center of $\mathcal{C}_s$ moves along $\gamma$ (see Fig. \ref{fig11} A).
\item  when $R\neq 1$, the third component of $\gamma$ makes the curve $\gamma$ as a helix and the surface $\Sigma$ is a tube-shaped obtained by moving the ellipse $\mathcal{C}_s$ along $\gamma$ (see Fig. \ref{fig11} B).
\end{enumerate}
\begin{figure}[ht!]
\minipage{0.48\textwidth}
  \includegraphics[width=\linewidth]{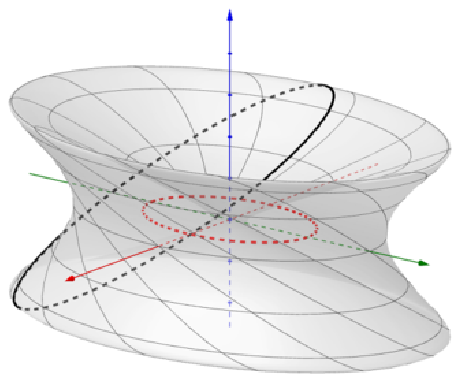}
  \caption*{A. $\gamma$ is a closed loop \\$(a=2.3, b=0.8)$}
\endminipage\hfill
\minipage{0.48\textwidth}%
  \includegraphics[width=\linewidth]{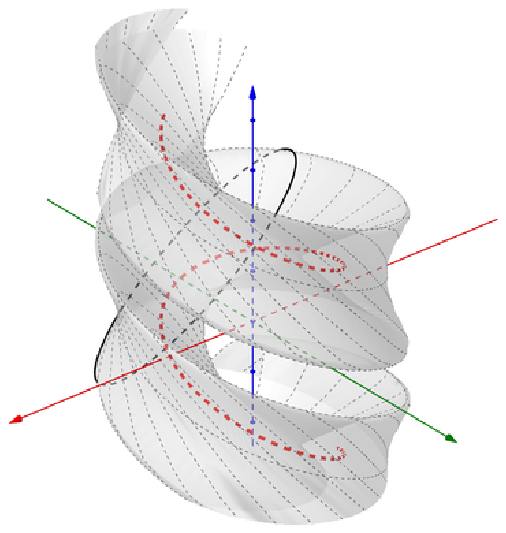}
  \caption*{B. $\gamma$ is not closed \\$(a=1.2, b=0.8)$}
\endminipage
\caption{The surfaces satisfy Pappus-Guldin theorem for areas (Example \ref{examplereqornotone}).}\label{fig11}
\end{figure}
\end{example}

\begin{example}\label{corollaryexample2}

For any numbers $a\in (-1,1)$ and $b \in \mathbb{R}$ suppose
$$\gamma_{a,b}(s)=\big(\sqrt{1-a^2} \ s, as +b , (1+b\sqrt{1-a^2}) s\big)$$
is the horizontally regular curve with horizontal arc-length. Since the projection of $\gamma$ onto the $xy$-plane is a line with slope $\frac{a}{\sqrt{1-a^2}}$, the p-curvature $\kappa\equiv 0$. We also have the contact normality $\tau\equiv 1$. Set $f(t)=t$, $g(t)\equiv 0$ and $h(t)=\sin(t)$ for $t\in [0, 2\pi]$, then those functions satisfy the assumptions in Corollary \ref{pappsarea} (2). In this case, we have the vectors $U(s)=\big(\sqrt{1-a^2}, a, b\sqrt{1 -a^2}\big)$, $V(s)=\big(-a, \sqrt{1-a^2}, -ab-s\big)$, and so the parametric surface is given by
\begin{align*}
\mathbb{X}(s,t)&=\gamma_{a,b}+fU+hT \\
&=\Big(\sqrt{1-a^2} (s+t), a(s+t)+b, b \sqrt{1-a^2}(s+t)+s+\sin(t) \Big).
\end{align*}
Note that the vector $U$ only depends on the number $a$ and the vector $V$ does not contribute to the surface, and so the surface $\Sigma$ is generated by the plane curve $\mathcal{C}_s: fU+hT$ along the line $\gamma_{a,b}$ as shown in Fig. \ref{fig2}.

\begin{figure}[ht!]
  \includegraphics[width=0.5 \textwidth]{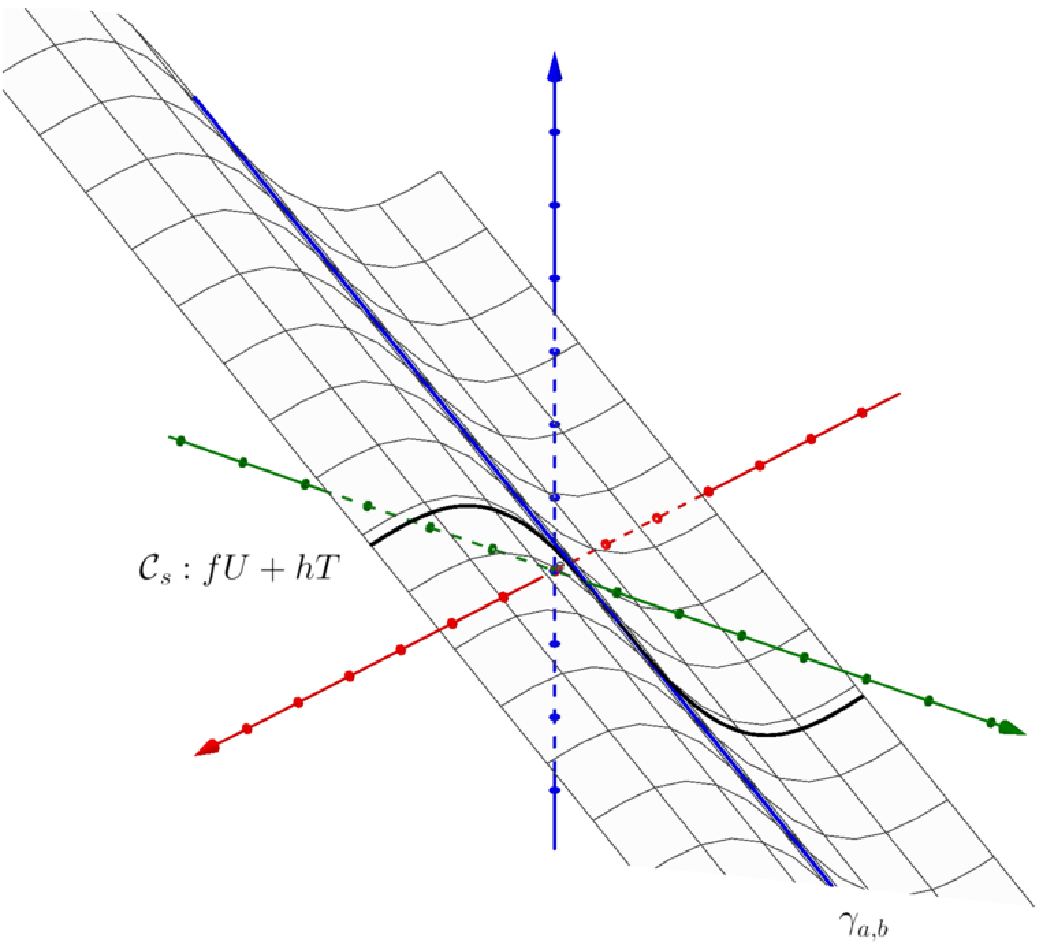}\label{fig21}
  \caption{A plane satisfies Pappus-Guldin theorem for areas (Example \ref{corollaryexample2} $a=-0.9, b=0.5$).}\label{fig2}
\end{figure}
\end{example}


In \cite{CHL} the author and the coauthors introduced the concept of the \emph{normal coordinates}. More precisely, for any regular surface $\Sigma=\mathbb{X}(s,t)=\gamma(s)+f(t)U(s)+g(t)V(s)$, for any $(s,t)\in D\rightarrow  \mathbb{H}_1$ we can find the coordinates (called normal coordinates) such that the parametric equation of the surface satisfies that
\begin{enumerate}
\item $\mathbb{X}(D)$ is a surface without singular points,
\item $\mathbb{X}_s$ defines the characteristic foliation on $\mathbb{X}(D)$,
\item $|\mathbb{X}_s|=1$ on $D$ w.r.t. the Levi-metric.
\end{enumerate}
Next we ask if the surfaces in Corollary \ref{pappsarea} satisfying Pappus-Guldin theorem for areas can be the surface with constant p-mean curvature and the answer is positive. Recall that \cite[Sec. 2, p. 134]{CHMY} the p-mean curvature $H$ of any surface $\Sigma$ is defined to be the variation of the p-area along the vector field $e_2=Je_1$ where $e_1$ is the characteristic vector field $\Sigma$. The p-mean curvature can also be represented by the function $H$ satisfying $\nabla_{e_1}e_2=He_2$ (up to a sign) where $\nabla$ is the pseudo-hermitian connection. When using the normal coordinates, one has $H=\langle \mathbb{X}_{tt}, J\mathbb{X}_t\rangle$. Suppose $\gamma(s)$ is parametrized by the horizontal arc-length and $h(t)\equiv 0$. Then the surface in Corollary \ref{pappsarea} is given by $\mathbb{X}(s,t)=\gamma+fU+gV$ and we have the derivatives
\begin{align*}
\mathbb{X}_t&=f'U+g'V=(f'x'-g'y')\mathring{e}_1+(f'y'+g'x')\mathring{e}_2, \\
\mathbb{X}_{tt}&=(f''x'-g''y')\mathring{e}_1+(f''y'+g''y')\mathring{e}_2.
\end{align*}
Thus, the p-mean curvature
\begin{align}\label{hcondition}
H&=\langle \mathbb{X}_{tt},J \mathbb{X}_t\rangle\\
&=\Big((x')^2+ (y')^2 \Big)(-f''g'+g''f') \nonumber \\
&=-f''g'+g''f'. \nonumber
\end{align}
Therefore by adding the constraint, $-f''g'+g''f'=const.$, we obtain the constant p-mean curvature surface and we summarize the result as follows.

\begin{corollary}\label{minimalsurface}
When $h(t)\equiv 0$. For any functions $f(t), g(t), \kappa(s), \tau(s)$ satisfying $(f')^2+ (g')^2=1$ and $-f''g'+g''f'=C$ for some constant $C$, then the surface $\Sigma$ defined by
\begin{align}\label{hequi0}
\mathbb{X}(s,t)=\gamma(s)+f(t)U(s)+g(t)V(s)
\end{align} is a constant p-mean curvature with $H=C$ and its p-area is given by
\begin{align*}
\mathcal{A}(\Sigma)= \int_0^{s_0}\int_0^{t_0} |\tau(s)-g(t)| dt ds.
\end{align*}
In particular, if $\tau\equiv 1>g $ on $[0,s_0]\times [0,t_0]$ (resp. $\tau\equiv -1 <g$), then Pappus-Guldin theorem holds, namely,
\begin{align*}
\mathcal{A}(\Sigma)&= s_0t_0-s_0 \int_0^{t_0}g(t) dt.   \\
 (\text{resp. } &=s_0t_0+s_0 \int_0^{t_0} g(t) dt)
\end{align*}
\end{corollary}

Note that if the surface $\Sigma$ defined by \eqref{hequi0} is a p-minimal surface (namely, $H=0$), it does not bound any solid region. Indeed, by \eqref{hcondition} the condition for being a p-minimal surface is that $H=-f''g'+g''f'=0$, and the condition holds if and only if $f=C_1 g+C_2$ for some constants $C_1, C_2$. Thus, the curve $\mathcal{C}_s(t)=\gamma+f(t)U+g(t)V$ degenerates to a line segment such that the surface $\Sigma$ generated by $\mathcal{C}_s(t)$ does not bound any solid region. The result has been observed in \cite[Sec. 4]{CHMY}. The authors in that paper showed that the p-minimal surfaces are classical ruled surfaces with the ruled generated by Legendrian lines and clearly the ruled surfaces do not enclose any bounded region.

\bibliography{mybib}
\end{document}